\documentclass[12pt, twoside,a4paper, draft]{article}
\usepackage[utf8]{inputenc}
\newcommand{\CopyName}{Yu.\ Volkov$^1$,\ O.\ Volkov$^2$,\ N.\ Voinalovych$^3$} % Author&#39;s name
\newcommand{\NAME}{Yu.\ Volkov,\ O.\ Volkov,\ N.\ Voinalovych} %
\newcommand{\Year}{2025} % Publishing year (to fill in by the Editors)
\newcommand{\Volume}{} % Volume (to fill in by the Editors)
\newcommand{\Number}{} % Issue (to fill in by the Editors)
\newcommand{\Pages}{} % Pages (to fill in by the Editors)
\newcommand{\rightheadtext}{MIXED EXPONENTIAL STATISTICAL STRUCTURES} %please give short version of the title, not exceeding 40 symbols
\usepackage[english,russian,ukrainian]{babel}
\usepackage{ms_we4}

\renewcommand{\refname}{\refnam}
\vs
\newcommand{\tit}{MIXED EXPONENTIAL STATISTICAL STRUCTURES AND THEIR APPROXIMATION OPERATORS} % A title of the paper.
\date{}
\begin{document}
\hbox to \textwidth{\footnotesize\textsc Mатематичні Студії. 
\hfill

Matematychni Studii}
\vspace{0.3in}
\textup{\scriptsize{УДК 519.21}} \vs % Індекс УДК.
\markboth{{\NAME}}{{\rightheadtext}}
\begin{center} \textsc {\CopyName} \end{center}
\begin{center} \renewcommand{\baselinestretch}{1.3}\bf {\tit} \end{center}

\vspace{20pt plus 0.5pt} {\abstract{ \noindent Yu.\  Volkov,\ O.\  Volkov,\ N.\  Voinalovych.\ % Author&#39;s name (in English).
\textit{Mixed  Exponential  Statistical Structures  and  Their
Approximation  Operators,} \vspace{3pt} \ English  % A title of the paper (in English).

 % Abstract (in English)

The paper examines the construction and analysis of a new class of mixed exponential statistical structures that combine the properties of stochastic models and linear positive operators. The relevance of the topic is driven by the growing need to develop a unified theoretical framework capable of describing both continuous and discrete random structures that possess approximation properties.
The aim of the study is to introduce and analyze a generalized family of mixed exponential statistical structures and their corresponding linear positive operators, which include known operators as particular cases. We define auxiliary statistical structures $\mathbf{B}$ and $\mathbf{H}$ through differential relations between their elements, and construct the main Phillips-type structure. Recurrent relations for the central moments are obtained, their properties are established, and the convergence and approximation accuracy of the constructed operators are investigated.
The proposed approach allows mixed exponential structures to be viewed as a generalization of known statistical systems, providing a unified analytical and stochastic description. The results demonstrate that mixed exponential statistical structures can be used to develop new classes of positive operators with controllable preservation and approximation properties. The proposed methodology forms a basis for further research in constructing multidimensional statistical structures, analyzing operators in weighted spaces, and studying their asymptotic characteristics.

}} \vsk

\subjclass{41A36, 41A25, 60E05, 47D03} % 2020 AMS Mathematics Subject Classification.

\keywords{mixed exponential statistical structures; linear positive operators; Phillips operators; Durrmeyer operators; central moments; covariance characteristic; stochastic models} %Keywords of the paper (in English) separated by semicolons
%\doi{?}%(to fill in by the Editors)
\renewcommand{\refname}{\refnam}

\vskip10pt

% The text of the paper

\noindent\textbf{1. Introduction.} 
In modern probability theory and mathematical statistics, studies of linear positive operators associated with various classes of distributions of random variables play a significant role. Such operators have deep theoretical significance and are widely used in problems of function approximation, modeling of random processes, and construction of numerical methods. Of particular interest are operators generated by stochastic structures that generalize classical distributions — Binomial, Poisson, Negative Binomial, Gamma, and Beta distributions.

The problem of constructing and studying the properties of positive operators has a long history in mathematical analysis. This direction was initiated by the work of R.~S.~Phillips, which proposed an inversion formula for the Laplace transform
$
f(x) = \int_{0}^{\infty} e^{-xt} d\mu(t),
$
and showed the connection between semigroups of linear operators
$
T_t = \exp(tA)
$
and statistical structures~\cite{1}.

Further development of these ideas led to the emergence of operators of the J.L.Durrmeyer type~\cite{2}
$$
D_n(f; x) = \sum_{k=0}^{n} f_k \binom{n}{k} x^k (1-x)^{n-k}, \quad f_k = (n+1) \int_{0}^{1} f(t) \binom{n}{k} t^k (1-t)^{n-k} dt.
$$
They combine a discrete and an integral component and provide accurate approximation of continuous functions.

Within the branch of Durrmeyer-type operators and their extensions, modifications, in particular the beta-type operators have attracted considerable attention. Thus, in the work of Deo (2008), direct results were obtained for the beta-operator variant of Durrmeyer's operators~\cite{3}:
$$
B_n^{(\alpha, \beta)}(f; x) = (n + \alpha + \beta + 1) \sum_{k=0}^{n} f_k B(k + \alpha + 1, n - k + \beta + 1),
$$
where $B(\cdot, \cdot)$ is the beta function.

Subsequently, numerous variations of such operators emerged — Bernstein-Durrmeyer, Szász-Durrmeyer, Beta-Durrmeyer, which became basic tools in the study of the properties of positive linear operators~\cite{4,5,6}.
Works devoted to the quantitative estimates of convergence for operators that preserve certain functions play a special role. In particular, in Bigou (2019)~\cite{7}, inequalities of the type
$
\|L_n(f) - f\| \leq C \omega \left(f; \frac{1}{\sqrt{n}}\right),
$
where $\omega(f; \delta)$ is the modulus of continuity, were obtained, which significantly expands the class of known estimates.
These results significantly expand the class of known convergence estimates and create a basis for the construction of new generalized operators, which include the mixed exponential structures considered in this work.

In parallel, another direction is developing — the modeling of stochastic processes using mixed exponential distributions, which combine several exponential components and are used in reliability statistics, risk analysis, and queuing theory~\cite{8}. Such models are distinguished by flexibility and analytical convenience, but the question of their integration with the operator approach remains open.

Thus, despite significant progress in the study of positive operators and mixed exponential distributions, the task of constructing a unified theoretical scheme that combines these directions remains topical. It is precisely such a scheme — in the form of mixed exponential statistical structures — that is proposed in this work.

The research method is the introduction and study of a generalized family of mixed exponential statistical structures and their corresponding linear positive operators, which encompass known operators as special cases.

\vspace{0.5em}
\noindent\textbf{2. Statistical Structures B.}
Let us denote by $b_{n,k} (x)$, $k = 0, 1, 2, \dots$ an integer-valued statistical structure that depends on parameters $x \in X, n \in N$, and is such that
\begin{equation}\label{eq1}
    b(x) \frac{d}{dx} b_{n,k} (x) = (k - nx) b_{n,k} (x),
\end{equation}
where $b(x)$ is a non-negative function, which we shall call the covariance characteristic of the structure. We shall call this structure the structure \textbf{B}.
\vspace{0.5em} % Невеликий відступ

\begin{remark*}
{\sl Not every non-negative function $b(x)$ can be the covariance characteristic of structure \textbf{B}; for example, the functions $b(x) = 1, ~b(x) = x^2$ cannot be.}
\end{remark*}
The structure \textbf{B} can be constructed as follows. Let us take the power series
$$
    \omega(y) = \sum_{k=0}^{\infty} a_k y^k
$$
with non-negative coefficients, $ a_k \ge 0, \ 0 \le y < R, \ R$~--- the radius
of convergence of the series. Let us consider a random variable $\xi$, which can take
non-negative integer values with probabilities
\begin{equation}\label{eq2}
P\{\xi = k\} := b_k \frac{y^k}{(\omega(y))^n}, \ k=0, 1, 2, \dots,
\end{equation}
where $b_k$ are the coefficients of the power series expansion of the function $(\omega(y))^n, ~n \in N$.

The sequence~\eqref{eq1} defines a statistical structure with parameters $y$ and $n$,
where
$$
    \operatorname{M} \xi = n y \frac{\omega'(y)}{\omega(y)}, \text{~and~} \operatorname{D} \xi = n y \frac{d \operatorname{M} \xi}{dy}.
$$
Let us denote by $y(x)$ the function inverse to the function (it exists because $\operatorname{D} \xi \ge 0$)
\begin{equation}\label{eq3}
x = y \frac{\omega'(y)}{\omega(y)},
\end{equation}
and by $b(x) := \frac{y(x)}{y'(x)}$. Then the mathematical expectation $\operatorname{M} \xi = nx$, and the variance
$\operatorname{D} \xi = nb(x)$. As a result, we obtain a family of distributions that will depend on
the parameter $x$:
$$
    b_{n,k} (x) := b_k \frac{(y(x))^k}{(\omega(y(x)))^n} \cdot
$$
Let us prove that the functions $b_{n,k}(x)$ satisfy the relation~\eqref{eq1}.
We have:
$$
    \log b_{n,k} (x) = \log b_k + k \log y(x) - n \log \omega(y(x)).
$$
Therefore
$$
    \frac{b'_{n,k} (x)}{b_{n,k} (x)} = k \frac{y'(x)}{y(x)} - n \frac{y'(x) \omega'(y(x))}{\omega(y(x))} = \frac{k}{b(x)} - n \frac{y'(x)}{y(x)} \frac{y(x) \omega'(y(x))}{\omega(y(x))},
$$
and taking into account~\eqref{eq3} we obtain~\eqref{eq1} from here.

\begin{lemma}
{\sl Let $I(x)$ be the Fisher information of structure \textbf{B}. Then
$$
I(x) = \frac{n}{b(x)}.
$$}
\end{lemma}

\begin{proof}[Proof]
Indeed,
$$
I(x) = \sum_{k=0}^{\infty} \left( \frac{d \log b_{n,k} (x)}{dx} \right)^2 b_{n,k} (x) = (b(x))^{-2} \sum_{k=0}^{\infty} (k - nx)^2 b_{n,k} (x) = \frac{n}{b(x)}.
$$
\end{proof}

\begin{lemma} {\sl Let $\beta_m (x)$ be the central moments of the distribution. Then the following recurrence relation holds:
\begin{equation}\label{eq4}
\beta_{m+1} (x) = b(x) \left( \frac{d \beta_m (x)}{dx} + n m \beta_{m-1} (x) \right), \ \beta_0 (x) = 1, \ \beta_1 (x) = 0.
\end{equation}}
\end{lemma}

\begin{proof}[Proof]
Indeed, $\beta_m (x) = \sum_{k=0}^{\infty} b_{n,k} (x) (k - nx)^m$. Then

$$
    \frac{d \beta_m (x)}{dx} = \sum_{k=0}^{\infty} \frac{d b_{n,k} (x)}{dx} (k - nx)^m - \sum_{k=0}^{\infty} m n b_{n,k} (x) (k - nx)^{m-1} =
$$
$$
    = \sum_{k=0}^{\infty} (b(x))^{-1} b_{n,k} (x) (k - nx)^{m+1} - \sum_{k=0}^{\infty} m n b_{n,k} (x) (k - nx)^{m-1},
$$
and \eqref{eq4} follows from here.
\end{proof}

\begin{corollary}
{\sl
$$
\beta_m (x) =
\begin{cases}
  c_m n^r (b(x))^r + O(n^{r-1}), & \text{if } m = 2r \\
  c_m n^r (b(x))^r b'(x) + O(n^{r-1}), & \text{if } m = 2r + 1
\end{cases},
$$
where $c_m =
\begin{cases}
  (2r - 1)!!, & \text{if } m = 2r \\
  \frac{1}{2} (2r)!! \sum\limits_{i=0}^{r-1} \frac{(2i+1)!!}{(2i)!!}, & \text{if } m = 2r + 1
\end{cases}.
$}
\end{corollary}
\vspace{0.5em}

\noindent
\textbf{Example 1.} The function $\omega(y) = 1 + y$ generates the structure
$$
  b_{n,k} (x) = C_n^k x^k (1-x)^{n-k}, \ 0 \le x \le 1, \ b(x) = x(1-x),
$$
i.e., the Binomial distribution with parameters $n$ and $x$.

\vspace{0.5em}
\noindent\textbf{Example 2.} The function $\omega(y) = e^y$ generates the structure
$$
  b_{n,k} (x) = e^{-nx} \frac{(nx)^k}{k!}, \ 0 \le x < \infty, \ b(x) = x,
$$
i.e., the Poisson distribution with parameter $nx$.

\vspace{0.5em}
\noindent\textbf{Example 3.} The function $\omega(y) = \frac{1}{1-y}$ generates the structure
$$
  b_{n,k} (x) = C_{n+k-1}^k x^k (1+x)^{-n-k}, \ 0 \le x < \infty, \ b(x) = x(1+x),
$$
i.e., the Negative Binomial distribution with parameters $n$ and $x$.

\vspace{0.5em}
\noindent\textbf{Example 4.} The function
$$
  \omega(y) = \frac{1 - \sqrt{1-4y}}{2y}
$$
generates the structure
$$
  b_{n,k} (x) = \frac{n}{2k+n} C_{2k+n}^k x^k (1+x)^{n+k} (1+2x)^{-n-2k}, ~\ 0 \le x < \infty,
$$
$$
  \ b(x) = x(1+x)(1+2x),
$$
i.e., the Catalap distribution with parameters $n$ and $x$ (see ~\cite[ p.~31]{9}).

% Заголовок розділу, відформатований як у прикладі

\vspace{0.51em}
\noindent\textbf{2. Statistical Structures H.}
Let us consider another auxiliary statistical structure, which is defined
by the family of densities $h_{n,k} (t), k = 0, 1, 2, \dots$, with parameters $k \in N, n \in N$, and such that
\begin{equation}\label{eq5}
    h(t) \frac{d}{dt} h_{n,k} (t) = (k - nt) h_{n,k} (t),
\end{equation}

\noindent
where $h(t)$ is a non-negative function, which we shall call the covariance
characteristic of the structure. We shall call this structure the structure \textbf{H}.

The structure \textbf{H} can be constructed as follows. Let us take a measure
$\mu(t)dt$ (absolutely continuous with respect to the Lebesgue measure) for which the
Laplace transform $u(s) = \int_{R} e^{-s\tau} \mu(\tau) d\tau$ exists. Let $t = - \frac{u'(s)}{u(s)}$ and let us denote by
$s(t)$ the inverse function to $t = t(s)$, and by $h(t) := - \frac{1}{s'(t)}$. Then
$$
    h_{n,k} (t) = c_{n,k} e^{-s(t)k} (u(s(t)))^{-n} \text{, ~where } c_{n,k} = \left( \int_{R} e^{-s(t)k} (u(s(t)))^{-n} dt \right)^{-1}
$$
(if the integral converges).

Let us prove that the functions $h_{n,k} (x)$ satisfy the relation~\eqref{eq1}.
We have:
$$
    \log h_{n,k} (t) = \log c_{n,k} - k s(t) - n \log u(s(t)). \text{ Therefore } \frac{h'_{n,k} (t)}{h_{n,k} (t)} = -k s'(t) - n s'(t) \frac{u'(s(t))}{u(s(t))},
$$
and \eqref{eq5} follows from here.

\begin{lemma}
{\sl The structure \textbf{H} exists if $h(t) = a t^2 + b t + c$, and the vector $(a, b, c)$ can take the following values: $(-1, 1, 0)$, $~(1, 1, 0)$, $~(0, 1, 0)$, $~(1, 0, 0)$, $~(0, 0, 1)$, $~(1, 0, 1)$.}
\end{lemma}

\begin{proof}[Proof]
We simply point out these structures, for which the relation
(5) is verified directly.

\noindent
{Structure (-1,1,0):}
$$
    h_{n,k} (t) = (n+1) C_{n}^k t^k (1-t)^{n-k}, \ k = 0, 1, 2, \dots, 0 \le t \le 1, \ h(t) = t(1-t).
$$

\noindent
{Structure (1,1,0):}
$$
    h_{n,k} (t) = (n-1) C_{n+k}^k t^k (1+t)^{-n-k}, \ k = 0, 1, 2, \dots, 0 < t < \infty, \ h(t) = t(1+t), \ n > 2.
$$

\noindent
{Structure (0,1,0):}
$$
    h_{n,k} (t) = \frac{e^{-nt} n^{k+1} t^k}{k!}, \ k = 0, 1, 2, \dots, 0 < t < \infty, h(t) = t.
$$

\noindent
{Structure (1,0,0):}
$$
    h_{n,k} (t) = \frac{k^{n-1} t^{-n}}{\Gamma(n-1)} e^{-k/n}, \ k = 0, 1, 2, \dots, 0 < t < \infty, \ h(t) = t^2, \ n > 2.
$$

\noindent
{Structure (0,0,1):}
$$
    h_{n,k} (t) = \sqrt{\frac{n}{2\pi}} \exp \left(- \frac{(k - n t)^2}{2n} \right), \ k = 0, 1, 2, \dots, - \infty < t < \infty, \ h(t) = 1.
$$

\noindent
{Structure (1,0,1):}
$$
    h_{n,k} (t) = c_{n,k} \exp(k \arctg(t)) (1+t^2)^{-n/2}, \ k = 0, 1, 2, \dots, n > 2,
$$
$$
    - \infty < t < \infty, \ h(t) = 1+t^2,
$$
where
$$
    (c_{n,k})^{-1} = \frac{2\operatorname{sh}(k\pi / 2)(2m - 2)!}{(k^2 + 4)(k^2 + 16)\cdots(k^2 + (2m - 2)^2)}, \quad \text{if } n = 2m
$$
$$
    (c_{n,k})^{-1} = \frac{2\operatorname{ch}(k\pi / 2)(2m)!}{(k^2 + 1)(k^2 + 9)\cdots(k^2 + (2m - 1)^2)}, \quad \text{if } n = 2m+1
$$
\end{proof}

\begin{remark*} {\sl If the argument $t$ of the density $h_{n,k}(t), k=0,1,2,\dots$ is specified
on an interval different from the entire number line, it is assumed that outside this
interval the density is zero.}
\end{remark*}
Further we will use structures with a quadratic covariance characteristic.

\begin{lemma} {\sl If $\alpha_1$ denotes the initial moment of the distribution \textbf{\textit{H}},
then
\begin{equation*}
\alpha_1 = \frac{k+b}{n-2a}.
\end{equation*}}
\end{lemma}

\begin{proof}[Proof] Let us rewrite the relation~\eqref{eq5} as:
\begin{equation*}
n t h_{n,k}(t) = k h_{n,k}(t) - (a t^2 + b t + c) h'_{k,t}(t)
\end{equation*}
and integrate. We get
$$
    n \alpha_1 = k - \int_{-\infty}^{\infty} (a t^2 + b t + c) h'_{k,t}(t) dt =
$$
$$
    = k - \left( \left[ (a t^2 + b t + c) h_{k,t}(t) \right]_{-\infty}^{\infty} - \int_{-\infty}^{\infty} (2 a t + b) h_{n,k}(t) dt \right) = k + 2 a \alpha_1 + b.
$$
\end{proof}

\noindent
The statement of the lemma follows from here.

\begin{lemma}{\sl If $\nu_m(x)$ denotes the central moments of the distribution \textbf{\textit{H}},
then the following recurrence relation holds:
\begin{equation}\label{eq6}
\nu_m(k) = \frac{1}{n-a(m+1)}\left(((m(2a\alpha_1+b)-(n\alpha_1-k))\nu_{m-1} + (m-1)(a\alpha_1^2 + b\alpha_1 + c)\nu_{m-2}\right),
\end{equation}
$$
    \nu_0 = 1, \nu_1 = 0.
$$}
\end{lemma}
\begin{proof}[Proof] Let us rewrite the relation~\eqref{eq5} as:
$$
    (n\alpha_1 - k)h_{n,k}(t) + n(t-\alpha_1)h_{n,k}(t) = -(at^2+bt+c)h'_{k,t}(t) =
$$
$$
    = -(a\alpha_1^2+b\alpha_1+c + (2a\alpha_1+b)(t-\alpha_1)+a(t-\alpha_1)^2)h'_{k,t}(t).
$$
Multiplying both sides of this equality by $(t-\alpha_1)^{m-1}$ and integrating, we obtain
$$
    (n\alpha_1 - k)\nu_{m-1} + n\nu_m = (at^2+bt+c)h_{k,t}(t)\big|_{-\infty}^{\infty} + (m-1)(a\alpha_1^2+b\alpha_1+c)\nu_{m-2} + a(m+1)\nu_m.
$$
\noindent
\eqref{eq6} follows from here.
In particular, if $m=2$, then
$$
    \nu_2 = \frac{(k+b)(ak+bn-ab)}{(n-2a)^2(n-3a)} + \frac{c}{n-3a}.
$$
\end{proof}

\noindent
\textbf{3. Structures of Phillips Type.}
\begin{definition} {\it A statistical structure of Phillips type} will be called the family of densities $p_n(x,t) = \sum_{k=0}^{\infty} b_{n,k}(x)h_{n,k}(t)$, which depends on the parameter $x, x \in X$, where $X$ is the image of the interval $[0,R)$ under the mapping $x=x(y)$ (see ~\cite{4}).
\end{definition}

Let the function $f$ be defined and integrable on the entire number line. Let us denote by $P_n(f,x)$ the mathematical expectation of the random variable $f(\eta)$, where the distribution of the random variable $\eta$ is given by the density $p_n(x,t)$, i.e.,
\begin{equation}\label{eq7}
    P_n(f,x)=\sum_{k=0}^\infty b_{n,k}(x)\int_{-\infty}^\infty f(t)h_{n,k}(t)dt .
\end{equation}
It is useful to view $P_n(f,x)$ as an operator; it is linear and positive. We shall call the function $b_{n,k}(x)$ the discrete kernel of the operator, and the function $h_{n,k}(t)$ the continuous kernel.

Thus, if $f(t)=t$, we obtain the mathematical expectation $\mathbf{M} \eta$, and in the case of the densities given above, $\mathbf{M} \eta = \alpha(x) = (nx+b)/(n-2a), n>2a$, indeed,

$$
    \alpha(x) = \sum_{k=0}^\infty b_{n,k}(x)\int_{-\infty}^\infty t h_{n,k}(t)dt = \sum_{k=0}^\infty b_{n,k}(x)\alpha_1 = \sum_{k=0}^\infty b_{n,k}(x)\frac{k+b}{n-2a} =
$$
$$
    = \sum_{k=0}^\infty b_{n,k}(x)\frac{(k-nx)+(nx+b)}{n-2a} = \sum_{k=0}^\infty b_{n,k}(x)\frac{k-nx}{n-2a} +
$$
$$
    + \sum_{k=0}^\infty b_{n,k}(x)\frac{nx+b}{n-2a} = 0 + \frac{nx+b}{n-2a}\cdot 1 = \frac{nx+b}{n-2a}.
$$
If $f(t)=(t - \mathbf{M} \eta)^m$, then $P_n(f,x) = \mu_m(x)$ is the central moment of the $m$-th order of the random variable $\eta$.

The paper investigates statistical structures of Phillips type and the operator $P_n(f,x)$.

\begin{theorem}\label{theor1} {\sl Let the covariance characteristic of structure $H$ be a polynomial of degree $r$. Then for the central moments of the random variable $\eta$ the following recurrence relation holds:
$$
    b(x)(\mu'_m(x) + m\alpha'(x)\mu_{m-1}(x)) = n\mu_{m+1}(x) + n(\alpha(x) - x)\mu_m(x) - \sum_{i=0}^r \frac{m+i}{i!}h^{(i)}(\alpha(x))\mu_{m+i-1}(x)
$$}
\end{theorem}

\begin{proof}[Proof]
First, let us establish the following auxiliary statement.

Let the function $f(t)$ be a polynomial of degree $r$, and the function $p(t)$ be continuously differentiable on the interval $[a,b]$ and $f(a)p(a) = f(b)p(b) = 0$.
Then
$$
    \int_a^b f(t)p'(t)(t-x)^m dt = -\sum_{i=0}^r (m+i) \frac{f^{(i)}(x)}{i!} \int_a^b p(t)(t-x)^{m+i-1} dt.
$$
Indeed, by Taylor's formula
$$
    f(t) = \sum_{i=0}^{r} \frac{f^{(i)}(x)}{i!}(t-x)^i.
$$
Therefore
$$
    \int_{a}^{b} f(t) p'(t) (t-x)^m dt = \int_{a}^{b} \sum_{i=0}^{r} \frac{f^{(i)}(x)}{i!} p'(t)(t-x)^{m+i} dt = \sum_{i=0}^{r} \frac{f^{(i)}(x)}{i!} \int_{a}^{b} p'(t)(t-x)^{m+i} dt =
$$
$$
    \sum_{i=0}^{r} \frac{f^{(i)}(x)}{i!} \left( p(t)(t-x)^{m+i} \Big|_a^b - (m+i)\int_{a}^{b} p(t)(t-x)^{m+i-1} dt \right) =
$$
$$
    \sum_{i=0}^{r} \frac{f^{(i)}(x)}{i!} \left( p(b)(b-x)^{m+i} - p(a)(a-x)^{m+i} - (m+i)\int_{a}^{b} p(t)(t-x)^{m+i-1} dt \right) =
$$
$$
    (b-x)^m p(b) f(b) - (b-x)^m p(b) f(b) - \sum_{i=0}^{r} \frac{f^{(i)}(x)}{i!} \left( (m+i)\int_{a}^{b} p(t)(t-x)^{m+i-1} dt \right) =
$$
$$
    - \sum_{i=0}^{r} \frac{f^{(i)}(x)}{i!} \left( (m+i)\int_{a}^{b} p(t)(t-x)^{m+i-1} dt \right).
$$
Since the central moments of the $m$-th order are found by the formula
$$
    \mu_m(x) = \sum_{k=0}^{\infty} b_{n,k}(x) \int_{-\infty}^{\infty} (t - \alpha(x))^m h_{n,k}(t) dt, \text{ we obtain from here}
$$
$$
    b(x)\mu'_m(x) = \sum_{k=0}^{\infty} b_{n,k}(x)(k-nx)\int_{-\infty}^{\infty} (t-\alpha(x))^m h_{n,k}(t)dt - mb(x)\alpha'(x)\mu_{m-1}(x),
$$
and from here
$$
    b(x)(\mu'_m(x) + m\alpha'(x)\mu_{m-1}(x)) = \sum_{k=0}^{\infty} b_{n,k}(x)\int_{-\infty}^{\infty} (k-nt)(t-\alpha(x))^m h_{n,k}(t)dt +
$$
$$
    + \sum_{k=0}^{\infty} b_{n,k}(x)\int_{-\infty}^{\infty} (n(t-\alpha(x)) + n(\alpha(x) - x))(t-\alpha(x))^m h_{n,k}(t)dt =
$$
$$
    \sum_{k=0}^{\infty} b_{n,k}(x) \int_{-\infty}^{\infty} h(t)(t-\alpha(x))^m h'_{n,k}(t) dt + n\mu_{m+1}(x) + n(\alpha(x) - x)\mu_m(x).
$$
It remains to use the auxiliary statement. We obtain
$$
    \int_{-\infty}^{\infty} h(t)(t-\alpha(x))^m h'_{n,k}(t) dt = - \sum_{i=0}^{r} \frac{(m+i)!}{i!} h^{(i)}(\alpha(x)) \int_{a}^{b} p_{n,k}(t)(t-\alpha(x))^{m+i-1} dt.
$$
\end{proof}
And the following statement follows from here.

\begin{corollary*}
{\sl If $h(t) = at^2 + bt + c$, then for the central moments of the Phillips structure we will have the following recurrence relation
$$
    (n - a(m+2))\mu_{m+1}(x) = b(x)\mu'_m(x) + \mu_m(x)((m+1)(2a\alpha(x) + b) - n(\alpha(x) - x)) +
$$
$$
    + \mu_{m-1}(x)(m(a(\alpha(x))^2 + b\alpha(x) + c) + mb(x)\alpha'(x)).
$$}
\end{corollary*}
In particular, if $m=1$, we obtain from here
$$
    \mu_2(x) = \frac{1}{n-3a} \left( a(\alpha(x))^2 + b\alpha(x) + c + b(x)\alpha'(x) \right) =
$$
$$
    = \frac{1}{n-3a} \left( a\left(\frac{nx+b}{n-2a}\right)^2 + b\frac{nx+b}{n-2a} + b(x)\frac{n}{n-2a}\right) , \quad n > 3a.
$$
\noindent
\textbf{Example 5.} Let the density be
$$
    h_{n,k}(t) = e^{-nt} \frac{n^{k+1} t^{k+1}}{k!}.
$$
Then the central moments $\mu_m(x)$ of the random variable $\eta$ satisfy the recurrence relation
\begin{gather}\label{eq8}
\mu_{m+1}(x) = \frac{1}{n} \left( b(x)\frac{d\mu_m(x)}{dx} + m\mu_m(x) + m\mu_{m-1}(x) \left( b(x) + x + \frac{1}{n} \right) \right),
~\mu_0 = 1, \mu_1 = 0.
\end{gather}
From here
$$
    \mu_2(x) = \frac{x+b(x)}{n} + \frac{1}{n^2}.
$$

\noindent
\textbf{Example 6.} Let the density be
$$
    h_{n,k}(t) = (n-1) C_{n+k+1}^{k} \frac{t^k}{(1+t)^{n+k}}.
$$
Then the central moments $\mu_m(x)$ of the random variable $\eta$ satisfy the recurrence relation
\begin{gather*}
    \mu_{m+1}(x) = \frac{1}{n-m-2} \biggl( b(x)\frac{d\mu_m(x)}{dx} + \frac{nm(2x+1)}{n-2}\mu_m(x) + \\
    +\left( \left(\frac{nx+1}{n-2}\right)^2 + \frac{nx+1}{n-2} + \frac{n}{n-2}b(x) \right) m\mu_{m-1}(x) \biggr), \\
    ~n > m+2, ~\mu_0 = 1, ~\mu_1 = 0.
\end{gather*}
From here
$$
    \mu_2(x) = \frac{1}{n-3} \left( \left(\frac{nx+1}{n- 2}\right)^2 + \frac{nx+1}{n-2} + \frac{n}{n-2}b(x) \right).
$$

\vspace{0.5em}
\noindent\textbf{4. Approximation Properties.}

\begin{theorem}\label{theor2} {\sl Let $f \in H^\omega$. Then for $n>3a$
$$
    |P_n(f, x) - f(x)| \leq
$$
$$
    \omega(1/\sqrt{n}) \left( 1 + \frac{n}{n-3a} \left( a\left(\frac{nx+b}{n-2a}\right)^2 + b\frac{nx+b}{n-2a} + b(x)\frac{n}{n-2a} \right) \right) + \omega\left( \left|\frac{2ax+b}{n-2a}\right| \right).
$$}
\end{theorem}

\begin{proof}[Proof]
$$
    |P_n(f, x) - f(x)| \leq |P_n(f, x) - f(\alpha(x))| + |f(\alpha(x)) - f(x)| \leq
$$
$$
    \leq \sum_{k=0}^{\infty} b_{n,k}(x) \int_{-\infty}^{\infty} |f(t) - f(\alpha(x))| h_{n,k}(t) dt + \omega(|\alpha(x) - x|) \leq
$$
$$
    \leq \sum_{k=0}^{\infty} b_{n,k}(x) \int_{-\infty}^{\infty} \omega(|t - \alpha(x)|) h_{n,k}(t) dt + \omega(|\alpha(x) - x|) \leq
$$
$$
    \leq \omega(\delta) \sum_{k=0}^{\infty} b_{n,k}(x) \int_{-\infty}^{\infty} \left( 1 + \left[ \delta^{-1} |t - \alpha(x)| \right] \right) h_{n,k}(t) dt + \omega(|\alpha(x) - x|) \leq
$$
$$
    \leq \omega(\delta) \left( 1 + \int_{t: \delta^{-1}|t-\alpha(x)| \geq 1} \left[ \delta^{-1}|t - \alpha(x)| \right] \left( \sum_{k=0}^{\infty} b_{n,k}(x) h_{n,k}(t) \right) dt \right) + \omega(|\alpha(x) - x|) \leq
$$
$$
    \leq \omega(\delta) \left( 1 + \sum_{k=0}^{\infty} b_{n,k}(x) \delta^{-2} \int_{-\infty}^{\infty} (t-\alpha(x))^2 h_{n,k}(t) dt \right) + \omega(|\alpha(x) - x|) =
$$
$$
    = \omega(\delta) (1 + \delta^{-2} \mu_2(x)) + \omega(|\alpha(x) - x|).
$$
If we now take $\delta = 1/\sqrt{n}$, we obtain the statement of the theorem.
\end{proof}

\begin{remark}{\sl If
$$
    b_{n,k}(x) = e^{-nx} (nx)^k / k!, ~\text {and} ~h_{n,k}(t) = e^{-nt} \frac{n^{k+1} t^{k+1}}{k!},
$$
then the operator $P_n(f, x)$ turns into the slightly modified Phillips operator~\cite{1}, which was used by him to obtain one of the formulas for the inversion of the Laplace transform.}
\end{remark}
This fact explains the name of the statistical structure under study and the corresponding operator.
For such operators, the inequality follows from Theorem 2
$$
    |P_n(f, x) - f(x)| \leq \omega(1/\sqrt{n})(1+x+x^2), ~x \geq 0.
$$

\begin{remark}{\sl There are a number of studies on the approximation properties of operators of type $P_n(f, x)$, where specific functions generated by distributions such as the Binomial, Negative Binomial, Gamma, and Beta distributions are taken as the functions $b_{n,k}(x)$ and $h_{n,k}(t)$. }
\end{remark}
For example, if the discrete kernel is taken as
$$
     b_{n,k}(x) = C_n^k x^k (1-x)^{n-k}, ~0 \leq x \leq 1,
$$
and the continuous kernel is taken as
$$
     h_{n,k}(t) = (n+1) C_n^k t^k (1-t)^k, ~ 0 \leq t \leq 1,
$$
then we obtain the Bernstein-Durrmeyer polynomials (see~\cite{2}).

For these polynomials, the inequality follows from Theorem 2
$$
     |P_n(f, x) - f(x)| \leq \frac{1}{4} \omega(1/\sqrt{n}) + \omega(1/n), ~0 \leq x \leq 1.
$$
If
$$
     b_{n,k}(x) = e^{-nx} \frac{(nx)^k}{k!}, ~ 0 \leq x < \infty, \quad h_{n,k}(t) = (n-1) C_{n+k}^k t^k (1+t)^{-n-k}, ~ 0 < t < \infty,
$$
then we obtain the Szász-Baskakov operators~\cite{4}; for such operators, the inequality follows from Theorem 2
\begin{gather*}
     |P_n(f, x) - f(x)| \leq \omega(1/\sqrt{n}) \left( 1 + \frac{n^2 x^2 + 2n^2 x - 2nx + n - 2}{(n-2)^2 (n-3)} n \right) + \omega\left( \frac{2x+1}{n-2} \right),\\~x \geq 0, ~n > 3.
\end{gather*}
If
\begin{gather*}
     b_{n,k}(x) = C_{n+k}^k x^k (1+x)^{-n-k}, ~ 0 < x < \infty, \\
     h_{n,k}(t) = (n-1) C_{n+k}^k t^k (1+t)^{-n-k}, ~0 < t < \infty,
\end{gather*}
then we obtain the Durrmeyer-Beta operators ~\cite{3}; for such operators, the inequality follows from Theorem 2
\begin{gather*}
     |P_n(f, x) - f(x)| \leq\\
     \leq\omega(1/\sqrt{n}) \left( 1 + \frac{n}{n-3} \left( \left(\frac{nx+1}{n-2}\right)^2 + \frac{nx^2 + 2nx + 1}{n-2} \right) \right) + \omega\left( \left|\frac{2x+1}{n-2}\right| \right), \\
     \quad x > 0, n > 3.
\end{gather*}

\noindent\textbf{5. Сonclusions.} 
The paper constructs a generalized family of mixed exponential statistical structures and their corresponding linear positive operators, which encompass well-known operators of Phillips and Durrmeyer types and their modifications as special cases. The proposed approach combines the statistical properties of mixed exponential distributions with operator methods from approximation theory, which made it possible to obtain new recurrence relations for the central moments and establish estimates of approximation accuracy.

It is shown that the introduced structures generalize classical statistical models (Binomial, Poisson, Beta, and Gamma distributions) and can be used to construct a wide range of operators with specified preservation and convergence properties. The obtained results create a basis for further research in the direction of constructing new operators with controllable statistical characteristics and applications in approximation theory, mathematical statistics, and stochastic process modeling.

Prospects for further research
lie in expanding the proposed approach to cases of multivariate statistical structures, constructing operators with weight functions, and analyzing their properties in weighted spaces. Of particular interest is the application of mixed exponential structures to problems of stochastic modeling, in particular for describing complex systems with components of different nature and investigating their approximation and statistical characteristics.

{\footnotesize

\vsk

${^1}$Volodymyr Vynnychenko Central Ukrainian State University

Kropyvnytskyi, Ukraine 

yuriivolkov38@gmail.com
\\

$^2$University of California at Berkeley

Berkeley, USA

oleksandr\_volkov@berkeley.edu
\\

${^3}$Volodymyr Vynnychenko Central Ukrainian State University

Kropyvnytskyi, Ukraine 

vojnalovichn@gmail.com

%\% Author's institution address, e-mail address.

\vs
}
%\received{?} \%(to fill in by the Editors)

%%\revised{?} \% (to fill in by the Editors)


\begin{thebibliography}{10}
\bibitem{1}
R.S. Phillips, {\it An inversion formula for Laplace transforms and semi-groups of linear operators}, 
Annals of Mathematics, {\bf 59} (1954), 325--356.

\bibitem{2}
J.L. Durrmeyer, {\it Une formule d’inversion de la transformée de Laplace: Application à la théorie des moments}, 
Thèse, Université de Paris, 1967.

\bibitem{3}
N. Deo, {\it Direct Result on the Durrmeyer Variant of Beta Operators}, 
Southeast Asian Bulletin of Mathematics, {\bf 32} (2008), 283--290.

\bibitem{4}
V. Gupta, V. Vasishtha, M.K. Gupta, {\it Rate of convergence of summation-integral type operators with derivatives of bounded variation}, 
Journal of Inequalities in Pure and Applied Mathematics, {\bf 4} (2) (2003), Article~34.

\bibitem{5}
M. Mursaleen, A.A.H. Alabied, {\it Approximation properties for modified (p,q)-Bernstein–Durrmeyer operators}, 
Mathematica Bohemica, {\bf 143} (2) (2018), 173--188. 
DOI: 10.21136/MB.2017.0086-16.

\bibitem{6}
A. Kajla, D. Miclăuş, {\it Modified Bernstein–Durrmeyer Type Operators}, 
Mathematics, {\bf 10} (11) (2022), Article~1876. 
DOI: 10.3390/math10111876.

\bibitem{7}
M.-M. Birou, {\it Quantitative results for positive linear operators which preserve certain functions}, 
General Mathematics, {\bf 27} (2) (2019), 85--95. 
DOI: 10.2478/gm-2019-0017.

\bibitem{8}
J.A. Barahona, Y.M. Gómez, E. Gómez-Déniz, O. Venegas, H.W. Gómez, 
{\it Scale Mixture of Exponential Distribution with an Application}, 
Mathematics, {\bf 12} (1) (2024), Article~156. 
DOI: 10.3390/math12010156.

\bibitem{9}
 
Yu.I. Volkov, {\it Positive Operators. Approximation. Probability}, 
NMK VO Publishers, Kyiv, 1992. (in Ukrainian)

\end{thebibliography}
\end{document}